\documentclass[11pt, a4paper]{article}
\usepackage{graphicx}
\usepackage{amssymb}
\usepackage{stmaryrd}
\usepackage{enumitem}

\usepackage{geometry}
\usepackage{tikz}
\usepackage{float}
\usepackage{url}
\usepackage{mathtools}
\usepackage{amsmath,amsfonts,amsthm}
\usepackage{mathrsfs} 
\newcommand{\irr}{\operatorname{irr}}

\newtheorem{theorem}{Theorem}[section]
\newtheorem{lemma}[theorem]{Lemma}
\newtheorem{proposition}{Proposition}[section]
\newtheorem{corollary}{Corollary}[section]

\usepackage{subfig}

\DeclarePairedDelimiter\ceil{\lceil}{\rceil}
\DeclarePairedDelimiter\floor{\lfloor}{\rfloor}

\usepackage[pagewise]{lineno}

\usepackage{ytableau}
\usepackage[colorlinks,
linkcolor=blue,
anchorcolor=blue,
citecolor=blue
]{hyperref}

\usepackage{yhmath}
\usepackage{mathdots}
\usepackage{MnSymbol}

\begin{document}
\begin{center}
{\large \bf Bounds on the Albertson Index for Trees with Given Degree Sequences}
\end{center}
\begin{center}
 Jasem Hamoud$^1$ \hspace{0.2 cm}  Duaa Abdullah$^{2}$\\[6pt]
 $^{1,2}$ Physics and Technology School of Applied Mathematics and Informatics \\
Moscow Institute of Physics and Technology, 141701, Moscow region, Russia\\[6pt]
Email: 	 $^{1}${\tt khamud@phystech.edu},
 $^{2}${\tt abdulla.d@phystech.edu}
\end{center}
\noindent
\begin{abstract}
In this paper, we presents novel and sharp bounds on the Albertson index of trees, revealing deep connections between degree sequences and graph irregularity where the Albertson index of Caterpillar tree satisfy 
\[
\operatorname{irr}(G)=\left( {{d_n} - 1} \right)^2 + \left( {d_1 - 1} \right)^2 + \sum\limits_{i = 2}^{n - 1} {\left( {{d_i} - 1} \right)\left( {{d_i} - 2} \right)} +\sum_{i=1}^{n-1}|d_i-d_{i+1}|.
 \]
 We derive powerful inequalities that precisely characterize the minimum and maximum values of the Albertson index, incorporating intricate dependencies on vertex degrees, edge counts, and the average of elements in degree sequence $\mathscr{D}=(d_1,d_2,\dots,d_n)$ where $d_n\geqslant d_{n-1}\geqslant \dots\geqslant d_2\geqslant d_1$. Our results not only improve existing extremal bounds but also uncover striking relationships between the structure of trees and their irregularity measurements. These advances open new avenues for the analysis of graph irregularity and contribute essential tools for the study of degree-based topological indices in combinatorial graph theory. 
\end{abstract}

\noindent\textbf{AMS Classification 2010:} 05C05, 05C12, 05C20, 05C25, 05C35, 05C76, 68R10.

\normalsize
\section{Introduction}\label{sec1}

Throughout this paper. Let $G=(V,E)$ be a simple, connected graph, where $n=|V(G)|$, $m=|E(G)|$. The study of structural irregularity in graphs has garnered increasing attention in both theoretical and applied contexts, particularly in mathematical chemistry, network analysis, and algorithmic graph theory. A graph is said to be regular if all of its vertices have equal degrees; otherwise, it exhibits some degree of irregularity. Accurately quantifying this irregularity is essential for modeling complex systems and understanding the behavior of real-world networks.

One of the earliest and most influential measures of graph irregularity is the \emph{Albertson index}, introduced in~\cite{albertson1997irregularity,Ghalavand23Gutman}, and defined for a simple graph $G = (V,E)$ as:
$$
\mathrm{irr}(G) = \sum_{uv \in E(G)} |\deg_G(u) - \deg_G(v)|,
$$
where $\deg_G(u)$ denotes the degree of vertex $u$. This index captures local disparities between adjacent vertex degrees and has inspired numerous variants and generalizations~\cite{abdo2019graph,abdo2014total}. The extremal behavior~\cite{Dimitrov2023GaoLin} of degree irregularities in trees is characterized by trees that either minimize or maximize irregularity indices based on their degree sequences. Greedy trees~\cite{Dimitrov2023GaoLin} typically minimize irregularity measures like the $\sigma$-irregularity by distributing degrees evenly, whereas adopting or similarly structured extremal trees maximize irregularity by creating larger degree differences between adjacent vertices. Among~\cite{Cvetkovi2010Rowlinson, Fiedler1973,Munir2023Wus}  provides a comprehensive overview of spectral graph theory, including how irregularity affects eigenvalues like $\lambda$ where $\lambda$ considered as the spectral radius or algebraic connectivity and their role in network properties. In~\cite{Bollob1998b} covers extremal graph theory and discusses how irregularity influences graph properties like degree distributions and connectivity.

Two of the earliest established topological graph indices are the first and second Zagreb indices, introduced in 1972 by Gutman and Trinajstic~\cite{Gutman72Moc}. These indices are frequently cited because of their applications in QSAR and QSPR studies. Beyond general graphs, a growing body of literature has focused on trees with constrained degree sequences. Studies by Molloy and Reed~\cite{molloy1995critical} and Broutin and Marckert~\cite{broutin2012asymptotics} explored their asymptotic behavior, while Zhang et al.~\cite{zhang2012number} investigated realizability and enumeration under degree constraints. These results have proven instrumental in analyzing extremal irregularity measures in structured tree families.

Among such families, \emph{caterpillar trees}—characterized by a linear backbone with attached pendant vertices—have emerged as key models in chemical graph theory~\cite{zhang2022random}. Their tractable structure allows the derivation of explicit analytical expressions. While prior work~\cite{dorjsembe2022irregularity,oboudi2023sombor} has examined irregularity bounds in these trees, closed-form expressions for both the Albertson and Sigma indices remain scarce.

The aim of this paper is to establish new bounds, inequalities, and refined estimates for the Albertson index of trees in terms of their order, size, and degree sequences. It focuses on deriving minimum and maximum values as well as lower and upper bounds for the Albertson index, often presenting these results through propositions, lemmas, theorems, and corollaries that relate the index to structural parameters like vertex degrees, maximum and minimum degrees, and combinatorial expressions.

This paper is organized as follows. Through Section~\ref{sec3}, we explore bounds on the minimum Albertson index, $\irr_{\min}$, for a tree $T$ of order $n$ with degree sequence $\mathscr{D} = (d_1, d_2, \dots, d_n)$, where the degrees are ordered such that $d_n \geqslant d_{n-1} \geqslant \cdots \geqslant d_1$. In Section~\ref{sec5}, we presented the upper bounds of the maximum value of Albertson index by considering the value of Albertson index. Section~\ref{sec5} provide a study of the effect bounds on Albertson index. 

 \section{Preliminaries}\label{sec2}
In this section, we review the most important basic concepts that we will use in establishing this research. Yang J, Deng H., M.,~\cite{YangDeng2023} provide $1\leq p \leq n-3$. Then: $\Delta(G)=n-1$. Through~\cite{Mandal2022Prvanovic} consider For any connected simple graph $G$ has not satisfied with the null graphs and path graphs, the Albertson index satisfy $\irr(G^*)=\irr(G)+2$. Through this paper, we point out that all the trees included in the study are caterpillar trees, we consider caterpillar trees denoted by $\mathscr{C}(n, m)$, where $n$ is the number of backbone (or path) vertices and $m$ is the number of pendant vertices attached to each.  
For Caterpillar tree~\cite{HamoudwithDuaa} with path vertices with degrees $d_1,d_2,\dots , d_n$, we have:
\begin{equation}~\label{basiceqn1}
\irr(G)=\left( {{d_n} - 1} \right)^2 + \left( {d_1 - 1} \right)^2 + \sum\limits_{i = 2}^{n - 1} {\left( {{d_i} - 1} \right)\left( {{d_i} - 2} \right)} +\sum_{i=1}^{n-1}|d_i-d_{i+1}|.
\end{equation}
 According to~\eqref{basiceqn1}, we consider the Albertson index had given by Proposition~\ref{propresentedn1}, where the graph $G$ be a regular graph.
 \begin{proposition}[\cite{Ghalavand23Gutman}]~\label{propresentedn1}
Let $G$ be a simple graph with vertices set $V(G)=\{v_1,v_2,\dots,v_n\}$. Then, the Albertson index satisfy
\begin{equation}~\label{eq1propresentedn1}
\irr(G)=\sum_{\{u,v\}\subseteq V(G)}|\deg_G(u)-\deg_G(v)|, \quad \irr(G)=n\sum_{i=1}^{n}\deg_G(v_i)^2-4m^2.
\end{equation}
 \end{proposition}
 By computing constraints on all degree sequences of $n$-vertex graphs, the degree sequence that maximizes in~\eqref{baseicequn2} by considering~\eqref{eq1propresentedn1} this bound corresponds to a graph $G$ with a maximum irregularity. Using variance, Mandal et al. in \cite{Mandal2022Prvanovic} showed:
\begin{equation}~\label{baseicequn2}
\irr(G) \leq \sqrt{ m \left(\sum_{i=1}^n d_i^2\right) - 4 m^2 },    
\end{equation}
where $m$ is the number of edges.
Let $T$ be a tree with vertices set $V(T)=\{v_0,v_1,\dots, v_n\}$ where the degree of center vertex $v_0$ satisfy $\deg_T(v_0)=\deg_T(v_i)_{i>0}+1$ as we show that by Figure~\ref{fig01startreecond01}. Then, the estimated Albertson index: 228 as we show that.
\begin{figure}[H]
    \centering
\begin{tikzpicture}[scale=2.2]
\draw   (9.5,4.5)-- (8.5,5);
\draw   (9.5,4.5)-- (9.5,5);
\draw   (9.5,4.5)-- (10.5,4.5);
\draw   (9.5,4.5)-- (8.5,4.5);
\draw   (9.5,4.5)-- (10.5,3.5);
\draw   (9.5,4.5)-- (9.5,3.5);
\draw   (9.5,4.5)-- (10,3.5);
\draw   (9.5,4.5)-- (9,4);
\draw   (8.5,5)-- (8.228044041028733,5.504018652942894);
\draw   (8.5,5)-- (7.998355909222116,5.216908488184627);
\draw   (8.5,5)-- (8,5.5);
\draw   (8.5,5)-- (8,5);
\draw   (8.5,5)-- (8.5,5.5);
\draw   (8.5,5)-- (8.07143849661513,4.80451388789548);
\draw   (9,4)-- (8.415970694325054,3.7500365555105732);
\draw   (9,4)-- (8.426411063952628,3.8753209910414537);
\draw   (9,4)-- (8.426411063952628,3.6143117503521203);
\draw   (9,4)-- (8.5,3.5);
\draw   (9,4)-- (8.729181783152258,3.4733667603798803);
\draw   (9,4)-- (9,3.5);
\draw   (9.5,3.5)-- (9,3);
\draw   (9.5,3.5)-- (9.152016753068985,3.0035501271390803);
\draw   (9.5,3.5)-- (9.261640634158505,3.0192106815804403);
\draw   (9.5,3.5)-- (9.392145254503173,3.008770311952867);
\draw   (9.5,3.5)-- (8.8701267731245,2.993109757511507);
\draw   (9.5,3.5)-- (9.5,3);
\draw   (10,3.5)-- (10.25869593359177,3.008770311952867);
\draw   (10,3.5)-- (9.62183338630979,3.0035501271390803);
\draw   (10,3.5)-- (9.726237082585524,3.008770311952867);
\draw   (10,3.5)-- (9.851521518116407,2.9983299423252934);
\draw   (10,3.5)-- (10,3);
\draw   (10,3.5)-- (10.122971128433315,3.008770311952867);
\draw   (10.5,3.5)-- (10.655429979439564,3.029651051208013);
\draw   (10.5,3.5)-- (10.8381364479221,3.087073084159666);
\draw   (10.5,3.5)-- (10.942540144197833,3.1914767804353996);
\draw   (10.5,3.5)-- (10.984301622708127,3.337641955221426);
\draw   (10.5,3.5)-- (11,3.5);
\draw   (10.5,3.5)-- (10.5,3);
\draw   (10.5,4.5)-- (10.702411642763643,4.997660726005587);
\draw   (10.5,4.5)-- (11,5);
\draw   (10.5,4.5)-- (10.989521807521912,4.783633148640334);
\draw   (10.5,4.5)-- (10.989521807521912,4.345137624282254);
\draw   (10.5,4.5)-- (11,4);
\draw   (10.5,4.5)-- (10.728512566832576,4.00582561138612);
\draw   (9.5,5)-- (9.183337861951705,5.514459022570467);
\draw   (9.5,5)-- (9.329503036737734,5.504018652942894);
\draw   (9.5,5)-- (9.5,5.5);
\draw   (9.5,5)-- (9.726237082585524,5.52489939219804);
\draw   (9.5,5)-- (9.983934368306597,5.536462408713421);
\draw   (9.5,5)-- (9,5.5);
\draw   (8.5,4.5)-- (8.02445683329105,4.611367049785374);
\draw   (8.5,4.5)-- (7.9722549851531825,4.230293558378947);
\draw   (8.5,4.5)-- (7.9722549851531825,4.386899102792547);
\draw   (8.5,4.5)-- (8,4);
\draw   (8.5,4.5)-- (8.316787182863107,3.99016505694476);
\draw   (8.5,4.5)-- (8.5,4);
\begin{scriptsize}
\draw [fill=black] (9.5,4.5) circle (1.5pt);
\draw [fill=black] (8.5,5) circle (1.5pt);
\draw [fill=black] (9.5,5) circle (1.5pt);
\draw [fill=black] (10.5,4.5) circle (1.5pt);
\draw [fill=black] (8.5,4.5) circle (1.5pt);
\draw [fill=black] (10.5,3.5) circle (1.5pt);
\draw [fill=black] (9.5,3.5) circle (1.5pt);
\draw [fill=black] (10,3.5) circle (1.5pt);
\draw [fill=black] (9,4) circle (1.5pt);
\draw [fill=black] (8.228044041028733,5.504018652942894) circle (1.5pt);
\draw [fill=black] (7.998355909222116,5.216908488184627) circle (1.5pt);
\draw [fill=black] (8,5.5) circle (1.5pt);
\draw [fill=black] (8,5) circle (1.5pt);
\draw [fill=black] (8.5,5.5) circle (1.5pt);
\draw [fill=black] (8.07143849661513,4.80451388789548) circle (1.5pt);
\draw [fill=black] (8.415970694325054,3.7500365555105732) circle (1.5pt);
\draw [fill=black] (8.426411063952628,3.8753209910414537) circle (1.5pt);
\draw [fill=black] (8.426411063952628,3.6143117503521203) circle (1.5pt);
\draw [fill=black] (8.5,3.5) circle (1.5pt);
\draw [fill=black] (8.729181783152258,3.4733667603798803) circle (1.5pt);
\draw [fill=black] (9,3.5) circle (1.5pt);
\draw [fill=black] (9,3) circle (1.5pt);
\draw [fill=black] (9.152016753068985,3.0035501271390803) circle (1.5pt);
\draw [fill=black] (9.261640634158505,3.0192106815804403) circle (1.5pt);
\draw [fill=black] (9.392145254503173,3.008770311952867) circle (1.5pt);
\draw [fill=black] (8.8701267731245,2.993109757511507) circle (1.5pt);
\draw [fill=black] (9.5,3) circle (1.5pt);
\draw [fill=black] (10.25869593359177,3.008770311952867) circle (1.5pt);
\draw [fill=black] (9.62183338630979,3.0035501271390803) circle (1.5pt);
\draw [fill=black] (9.726237082585524,3.008770311952867) circle (1.5pt);
\draw [fill=black] (9.851521518116407,2.9983299423252934) circle (1.5pt);
\draw [fill=black] (10,3) circle (1.5pt);
\draw [fill=black] (10.122971128433315,3.008770311952867) circle (1.5pt);
\draw [fill=black] (10.655429979439564,3.029651051208013) circle (1.5pt);
\draw [fill=black] (10.8381364479221,3.087073084159666) circle (1.5pt);
\draw [fill=black] (10.942540144197833,3.1914767804353996) circle (1.5pt);
\draw [fill=black] (10.984301622708127,3.337641955221426) circle (1.5pt);
\draw [fill=black] (11,3.5) circle (1.5pt);
\draw [fill=black] (10.5,3) circle (1.5pt);
\draw [fill=black] (10.702411642763643,4.997660726005587) circle (1.5pt);
\draw [fill=black] (11,5) circle (1.5pt);
\draw [fill=black] (10.989521807521912,4.783633148640334) circle (1.5pt);
\draw [fill=black] (10.989521807521912,4.345137624282254) circle (1.5pt);
\draw [fill=black] (11,4) circle (1.5pt);
\draw [fill=black] (10.728512566832576,4.00582561138612) circle (1.5pt);
\draw [fill=black] (9.183337861951705,5.514459022570467) circle (1.5pt);
\draw [fill=black] (9.329503036737734,5.504018652942894) circle (1.5pt);
\draw [fill=black] (9.5,5.5) circle (1.5pt);
\draw [fill=black] (9.726237082585524,5.52489939219804) circle (1.5pt);
\draw [fill=black] (9.983934368306597,5.536462408713421) circle (1.5pt);
\draw [fill=black] (9,5.5) circle (1.5pt);
\draw [fill=black] (8.02445683329105,4.611367049785374) circle (1.5pt);
\draw [fill=black] (7.9722549851531825,4.230293558378947) circle (1.5pt);
\draw [fill=black] (7.9722549851531825,4.386899102792547) circle (1.5pt);
\draw [fill=black] (8,4) circle (1.5pt);
\draw [fill=black] (8.316787182863107,3.99016505694476) circle (1.5pt);
\draw [fill=black] (8.5,4) circle (1.5pt);
\end{scriptsize}
\end{tikzpicture}
    \caption{Star tree with determine condition}
    \label{fig01startreecond01}
\end{figure}
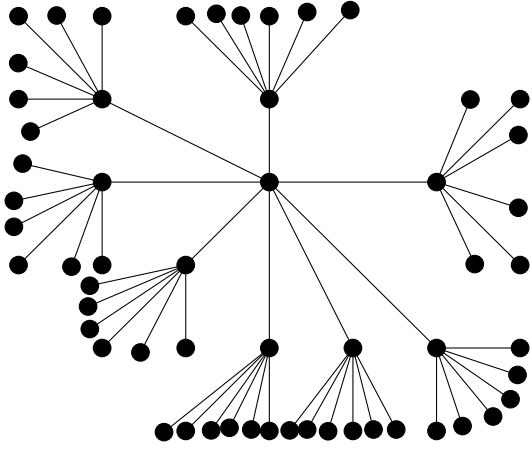
\begin{lemma}[\cite{HamoudwithDuaa}]~\label{le.alb5}
Let be $T$ tree of order $n\geqslant5$, a degree sequence $d=(d_1,d_2,d_3,d_4,d_5)$ where $d_5\geqslant d_4\geqslant d_3 \geqslant d_2 \geqslant d_1$, then Albertson index among tree $T$ is: 
\[
\irr(T)=d_1^2+d_n^2+\sum_{i=2}^{n-1}\lvert d_i-d_{i+1}\rvert+\sum_{i=2}^{4}(d_i+2)(d_i-1)-2.
\]
\end{lemma}
Theorem~\ref{perthmn1} gives an inequality that restricts the spread of degree differences in a connected nonregular graph, using the minimum degree $\delta$ and the maximum degree $\Delta$.
\begin{theorem}[\cite{Ghalavand23Gutman}]~\label{perthmn1}
Let $G$ be a connected nonregular graph. Then, the Albertson index satisfy
\[
\sum_{uv \in E(G)} (\Delta - \delta) | d_G(u) - d_G(v)| = (\Delta - \delta)\mathrm{irr}(G).
\]
\end{theorem}

\subsection{Statement of Problem}
This work focuses on establishing new bounds and inequalities for the Albertson index $\irr(T)$ of trees $T$ of order $n$, characterized by their degree sequences $\mathscr{D} = (d_1, d_2, \dots, d_n)$. The Albertson index, a topological graph index measuring the irregularity of a graph, plays a significant role in extremal graph theory and its applications.

We investigate both the minimum and maximum values of $\irr(T)$ for such trees, aiming to improve existing estimates and to describe extremal structures that achieve these bounds. Specifically, we study:

\begin{itemize}
  \item Lower bounds on the minimum Albertson index $\irr_{\min}$ involving the maximum degree $\Delta$ and sums of degree powers (Propositions~\ref{resseptn1}, \ref{resseptn2}, \ref{resseptn3}).
  \item Upper and lower bounds on the maximum Albertson index $\irr_{\max}$, including conditions depending on the degree sequence and maximum vertex degree (Propositions~\ref{maxresseptn1}, Lemma~\ref{lem1maxvalue}).
  \item Relationships between the minimum and maximum Albertson index values and their differences (Proposition~\ref{maxresseptn2}).
  \item New inequalities involving the Albertson index, vertex degrees, and other graph invariants expressed in Theorems~\ref{tythmAlbn1} to \ref{tythmAlbn4} and associated corollaries, which provide explicit algebraic lower and upper bounds on $\irr(T)$.
  \item Special cases and expressions for degree sequences with ordering $d_n \geq d_{n-1} \geq \dots \geq d_1$, as well as asymptotic and trigonometric forms related to the index.
\end{itemize}

The main goal is to contribute sharper estimates for $\irr(T)$ in terms of vertex degree parameters and to clarify how these bounds vary with structural properties of the tree, thereby enriching the theory of graph irregularity and its applications in network analysis.

\section{Bounds on The Minimum Value of Albertson Index}\label{sec3}
The study of irregularity measures in graphs provides valuable insights into their structural properties and embeddings. Among these, the Albertson index, which quantifies the minimum total imbalance across all possible planar embeddings of a graph, is particularly significant for understanding the combinatorial and geometric constraints of trees. In this section, we explore bounds on the minimum Albertson index, $\irr_{\min}$, for a tree $T$ of order $n$ with degree sequence $\mathscr{D} = (d_1, d_2, \dots, d_n)$, where the degrees are ordered such that $d_n \geqslant d_{n-1} \geqslant \cdots \geqslant d_1$. The maximum degree of the tree is denoted by $\Delta = d_n$.

We present two key propositions that establish bounds on $\irr_{\min}$ in terms of the degree sequence. Proposition~\ref{resseptn1} provides a normalized bound, demonstrating that the minimum Albertson index scales relative to $\Delta(\Delta-1)^2$, yielding a value strictly between 0 and 1. Proposition~\ref{resseptn2} offers a lower bound for $\irr_{\min}$, expressing it in relation to the sum of squared degrees and specific terms involving the smallest and largest degrees in the sequence. Together, these propositions deepen our understanding of the interplay between a tree’s degree distribution and its embedding irregularity, paving the way for further exploration of graph-theoretic properties in combinatorial settings.

\begin{proposition}~\label{resseptn1}
Let $T$ be a tree of order $n$ with degree sequence $\mathscr{D} = (d_1, d_2, \dots, d_n)$. Then the minimum value of the Albertson index satisfies
\begin{equation}~\label{eq1resseptn1}
    0 < \frac{2\irr_{\min}}{\Delta(\Delta - 1)^2} < 1 .
\end{equation}
\end{proposition}
\begin{proof}
Suppose the degree sequence $\mathscr{D} = (d_1, d_2, \dots, d_n)$ is ordered such that $d_n \geqslant d_{n-1} \geqslant \dots \geqslant d_2 \geqslant d_1$. Then, clearly $\Delta = d_n$ and $\irr_{\min} > 0$. Thus,
\begin{equation}~\label{eq2resseptn1}
\sqrt{\frac{n \Delta}{2}} < \irr_{\min}.
\end{equation}
Therefore,
\begin{equation}~\label{eq03resseptn1}
\irr_{\min} < n d_1 + \sum_{i=1,2,4} (d_i + 2)(d_i - 1).
\end{equation}
Note that since the degree sequence is increasing, inequality~\eqref{eq03resseptn1} holds for the two fundamental parameters $n$ and $\Delta$. Therefore, when taking the difference between them, relation~\eqref{eq3resseptn1} is satisfied by considering relation~\eqref{eq2resseptn1}. Thus,
\begin{equation}~\label{eq3resseptn1}
3 d_n^2 - \frac{\Delta (\Delta - 1)^2}{n} \leqslant 2 \irr_{\min}.
\end{equation}
From~\eqref{eq2resseptn1} and~\eqref{eq3resseptn1}, we conclude that~\eqref{eq1resseptn1} holds.
\end{proof}

Proposition~\ref{resseptn2} holds for trees where the Albertson index is significantly larger than $\sum d_i^2,$ possibly trees with high maximum degrees or specific structures. 

\begin{proposition}~\label{resseptn2}
Let $T$ be a tree of order $n$ with degree sequence $\mathscr{D} = (d_1, d_2, \dots, d_n)$. Then the minimum value of the Albertson index satisfies
\begin{equation}~\label{eq1resseptn2}
\irr_{\min} - \left(\sum_{i=1}^{n} d_i^2 + 2 d_1 + d_2 - 3 d_3 + 2(d_{n-1} + d_n)\right) \geqslant \frac{\Delta}{4}.
\end{equation}
\end{proposition}

\begin{proof}
Let $T$ be a tree of order $n$ with degree sequence $\mathscr{D} = (d_1, d_2, \dots, d_n)$ ordered such that $d_n \geqslant d_{n-1} \geqslant \dots \geqslant d_2 \geqslant d_1$. Then 
\begin{equation}~\label{eq01resseptn2}
\irr_{\min} \geqslant \sum_{i=1}^n d_i^2 + c,
\end{equation}
where $c$ is a correction term depending on the degrees. 

For the degree sequence in non-increasing order $\mathscr{D} = (\Delta, 1, 1, \dots, 1)$, we have
\[
\sum_{i=1}^n d_i^2 = \Delta^2 + \Delta.
\]

Also,
\begin{equation}~\label{eq2resseptn2}
\lambda = \sum_{i=1}^{n} d_i^2 + 2 d_1 + d_2 - 3 d_3 + 2(d_{n-1} + d_n) \leqslant 4 \Delta.
\end{equation}

Since relation~\eqref{eq2resseptn2} involves the degree sequence with the terms $2 d_1 + d_2 - 3 d_3$ and twice the sum of the last two degrees $2(d_{n-1} + d_n)$, we note that it is necessary that 
\[
2(d_{n-1} + d_n) > 2 d_1 + d_2 - 3 d_3 \quad \text{and} \quad d_{n-1} + d_n > \Delta
\]
to satisfy the relation
\begin{equation}~\label{eq3resseptn2}
\frac{\irr_{\min}}{2 d_1 + d_2 - 3 d_3 + 2 (d_{n-1} + d_n)} < d_n.
\end{equation}

Thus, from~\eqref{eq01resseptn2}, \eqref{eq2resseptn2}, and \eqref{eq3resseptn2}, inequality~\eqref{eq1resseptn2} holds.
\end{proof}

Let $\alpha=(d_1+d_2)/2$ and $\beta=(d_{n-1}+d_n)/2$, 
\begin{proposition}~\label{resseptn3}
Let $T$ be a tree of order $n$ with degree sequence $\mathscr{D}=(d_1,d_2,\dots,d_n)$ and integers $\alpha$ and $\beta$. Then 
\begin{equation}~\label{eq1resseptn3}
\frac{\beta m(m+1)}{\alpha} + \alpha \irr_{\min}(T) \geqslant \sum_{i=1}^n d_i^3.
\end{equation}
\end{proposition}
\begin{proof}
Assume a degree sequence $\mathscr{D} = (d_1,d_2,\dots,d_n)$ ordered such that $d_n \geqslant d_{n-1} \geqslant \dots \geqslant d_2 \geqslant d_1$. Then, for $T$ a connected graph with $n$ vertices and $m$ edges, suppose it satisfies $\irr_{\min} < m(m+1)$. Thus, the minimum value of the Albertson index satisfies
\[
2^\alpha m(m+1) \leqslant \alpha^4 \irr_{\min}(T).
\]
Then,
\begin{equation}~\label{eq2resseptn3}
\alpha \irr_{\min}(T) \geqslant \sum_{i=1}^n d_i^2.
\end{equation}
Therefore, the relation~\eqref{eq2resseptn3} grows by combining the term $2^\alpha m(m+1)$ where $\beta < 2^\alpha$, considering $\irr_{\min} < m(m+1)$. Thus, from~\eqref{eq2resseptn3} we find that 
\[
\irr_{\min} \leqslant n(2m - n + \Delta)
\]
where $0 < \irr_{\min} / 2^\alpha < n$. This condition is used to restrict the values of $\alpha$ in order to reach Equation~\eqref{eq1resseptn3}. Therefore, we clearly observe the impact of this condition on Equation~\eqref{eq3resseptn3}, as it provides a clear path to reach the desired result. Then,
\begin{equation}~\label{eq3resseptn3}
\frac{\beta m(m+1)}{\alpha} \geqslant \sum_{i=1}^n d_i^2, \quad \frac{\beta m(m+1)}{\alpha} + 2^\alpha m(m+1) \geqslant \alpha^4 \irr(T).
\end{equation}
Thus, from~\eqref{eq3resseptn3} we obtain
\begin{equation}~\label{eq4resseptn3}
\frac{\beta m(m+1)}{\alpha} + 2^\alpha m(m+1) \geqslant \sum_{i=1}^n d_i^3,
\end{equation}
by considering the term $2^{-\alpha} \irr_{\min} < n$ which satisfies $2^\alpha m(m+1) > \alpha \irr_{\min}(T)$. Then, from~\eqref{eq4resseptn3} and by considering~\eqref{eq3resseptn3} and~\eqref{eq2resseptn3}, we conclude that Equation~\eqref{eq1resseptn3} holds as desired.
\end{proof}

Through the previously presented study through Propositions~\ref{resseptn1}, \ref{resseptn2} and \ref{resseptn3}, we clarify these relationships and their impact using Table~\ref{tab01minimumalb} in order to calculate the minimum values of the Peterson index by providing an increasing nested sequence. Denote to $\frac{\beta m(m+1)}{\alpha}+\alpha \irr_{\min}(T)$ by $L_3$, $\irr_{\min} -\left(\sum_{i=1}^{n}d_i^2+2d_1+d_2-3d_3+2(d_{n-1}+d_n)\right)$ by $L_2$  and $\frac{2\irr_{\min}}{\Delta(\Delta-1)^2}$ by $L_1$.

\begin{table}[H]
\centering
\begin{tabular}{|l|c|c|c||l|c|c|c|}
\hline
 Degrees & $L_1$ & $L_2$  & $L_3$ &  Degrees & $L_1$ & $L_2$  & $L_3$ \\ \hline
$(2,4,5,7,9)$ & 0.6 & 4 & 2460 & $(3,6,7,10,12)$ & 0.5& 5 & 5092 \\ \hline
$(4,8,9,13,15)$ & 0.4 & 6 & 9052 & $(5,10,12,16,18)$ & 0.3 & 1 & 15046 \\ \hline
$(6,12,14,19,22)$ & 0.2 & 1 & 23528 & $(7,14,16,22,25)$ & 0.2 & 2 & 33265 \\ \hline
$(8,16,18,25,28)$ & 0.2 & 3 & 45328 & $(9,18,20,28,31)$ & 0.2 & 4 & 59961\\ \hline
\end{tabular}
\caption{Calculate the minimum values of Albertson index.}
\label{tab01minimumalb}
\end{table}

This values should help understand extremal bounds and behavior of irregularity indices in trees (see Figure~\ref{fig01tree}), providing both numeric evidence and theoretical confirmation of the indices’ dependency on the degree structure. The indices help measure the degree irregularity of the graph, quantifying how far the graph's degree distribution deviates from regularity.

\begin{figure}[H]
    \centering
\begin{tikzpicture}[scale=1]
\draw   (3,8)-- (2,7);
\draw   (3,8)-- (4,8);
\draw   (4,8)-- (3,7);
\draw   (4,8)-- (4,7);
\draw   (4,8)-- (5,8);
\draw   (5,8)-- (4.48,7);
\draw   (5,8)-- (5,7);
\draw   (5,8)-- (6,7);
\draw   (5,8)-- (7,8);
\draw   (7,8)-- (6.5,7);
\draw   (7,8)-- (7,7);
\draw   (7,8)-- (7.5,7);
\draw   (7,8)-- (8,7);
\draw   (7,8)-- (9,7);
\draw   (7,8)-- (9,8);
\draw   (9,8)-- (10,9);
\draw   (9,8)-- (9.5,9);
\draw   (9,8)-- (9.5,9.5);
\draw   (9,8)-- (9,9.5);
\draw   (9,8)-- (8.5,9.5);
\draw   (9,8)-- (8,9.5);
\draw   (9,8)-- (8,9);
\draw   (9,8)-- (10.42,8.62);
\draw   (2,5)-- (1,4);
\draw   (2,5)-- (2,4);
\draw   (2,5)-- (3,5);
\draw   (3,5)-- (3,4);
\draw   (3,5)-- (4,4);
\draw   (3,5)-- (4,5);
\draw   (3,5)-- (4,6);
\draw   (3,5)-- (3,6);
\draw   (3,5)-- (2,6);
\draw   (4,4)-- (3,3);
\draw   (4,4)-- (4,3);
\draw   (4,4)-- (5,3);
\draw   (4,4)-- (5,4);
\draw   (4,4)-- (5,5);
\draw   (4,4)-- (3.52,3.02);
\draw   (5,4)-- (6,5);
\draw   (5,4)-- (6,3);
\draw   (6,3)-- (7,4);
\draw   (6,3)-- (7,2);
\draw (-1.5,8.5) node[anchor=north west] {$\mathscr{D}=(2,4,5,7,9)$};
\draw (-1.5,5.5) node[anchor=north west] {$\mathscr{D}=(3,6,7,10,12)$};
\draw (5.886018366581596,4.472645160669947) node[anchor=north west] {$\vdots$};
\draw (6.899093356243413,3.4595701710081284) node[anchor=north west] {$\vdots$};
\draw (4.773622299501954,4.8) node[anchor=north west] {$10$};
\draw (5.4,3) node[anchor=north west] {$12$};
\begin{scriptsize}
\draw [fill= black] (3,8) circle ( 1.5pt);
\draw [fill= black] (2,7) circle ( 1.5pt);
\draw [fill= black] (4,8) circle ( 1.5pt);
\draw [fill= black] (3,7) circle ( 1.5pt);
\draw [fill= black] (4,7) circle ( 1.5pt);
\draw [fill= black] (5,8) circle ( 1.5pt);
\draw [fill= black] (4.48,7) circle ( 1.5pt);
\draw [fill= black] (5,7) circle ( 1.5pt);
\draw [fill= black] (6,7) circle ( 1.5pt);
\draw [fill= black] (7,8) circle ( 1.5pt);
\draw [fill= black] (6.5,7) circle ( 1.5pt);
\draw [fill= black] (7,7) circle ( 1.5pt);
\draw [fill= black] (7.5,7) circle ( 1.5pt);
\draw [fill= black] (8,7) circle ( 1.5pt);
\draw [fill= black] (9,7) circle ( 1.5pt);
\draw [fill= black] (9,8) circle ( 1.5pt);
\draw [fill= black] (10,9) circle ( 1.5pt);
\draw [fill= black] (9.5,9) circle ( 1.5pt);
\draw [fill= black] (9.5,9.5) circle ( 1.5pt);
\draw [fill= black] (9,9.5) circle ( 1.5pt);
\draw [fill= black] (8.5,9.5) circle ( 1.5pt);
\draw [fill= black] (8,9.5) circle ( 1.5pt);
\draw [fill= black] (8,9) circle ( 1.5pt);
\draw [fill= black] (10.42,8.62) circle ( 1.5pt);
\draw [fill= black] (2,5) circle ( 1.5pt);
\draw [fill= black] (1,4) circle ( 1.5pt);
\draw [fill= black] (2,4) circle ( 1.5pt);
\draw [fill= black] (3,5) circle ( 1.5pt);
\draw [fill= black] (3,4) circle ( 1.5pt);
\draw [fill= black] (4,4) circle ( 1.5pt);
\draw [fill= black] (4,5) circle ( 1.5pt);
\draw [fill= black] (4,6) circle ( 1.5pt);
\draw [fill= black] (3,6) circle ( 1.5pt);
\draw [fill= black] (2,6) circle ( 1.5pt);
\draw [fill= black] (3,3) circle ( 1.5pt);
\draw [fill= black] (4,3) circle ( 1.5pt);
\draw [fill= black] (5,3) circle ( 1.5pt);
\draw [fill= black] (5,4) circle ( 1.5pt);
\draw [fill= black] (5,5) circle ( 1.5pt);
\draw [fill= black] (3.52,3.02) circle ( 1.5pt);
\draw [fill= black] (6,5) circle ( 1.5pt);
\draw [fill= black] (6,3) circle ( 1.5pt);
\draw [fill= black] (7,4) circle ( 1.5pt);
\draw [fill= black] (7,2) circle ( 1.5pt);
\end{scriptsize}
\end{tikzpicture}
\caption{Trees had discussed according to Table~\ref{tab01minimumalb}.}
    \label{fig01tree}
\end{figure}
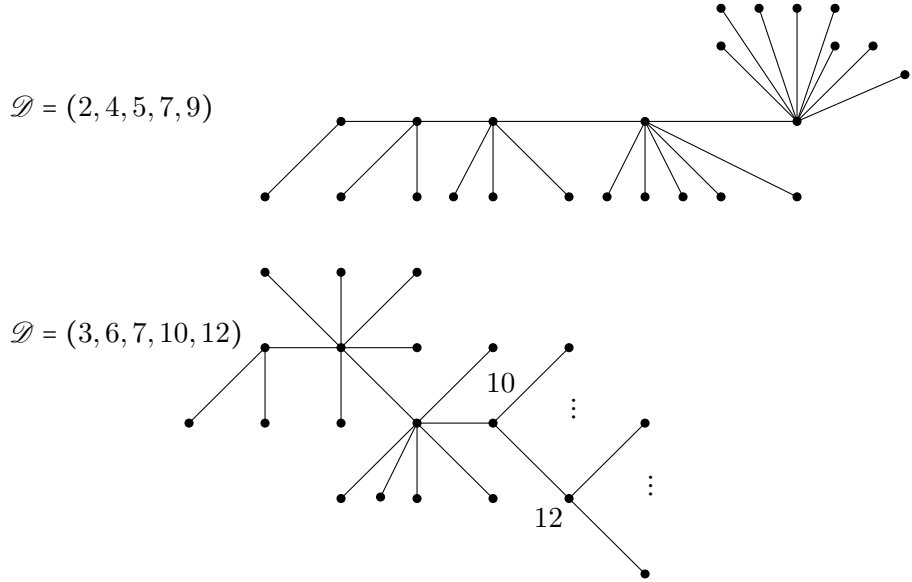

\section{Bounds on Albertson Index}\label{sec4}
According to Proposition~\ref{resseptn1}, we noticed that $2\irr_{\max}<\Delta(\Delta-1)^2,$ we presented Proposition~\ref{maxresseptn1} for given the upper bound of the maximum value of Albertson index.  
\begin{proposition}~\label{maxresseptn1}
Let $T$ be a tree of order $n$ with degree sequence $\mathscr{D}=(d_1,d_2,\dots,d_n)$. Then the maximum value of the Albertson index satisfies
\begin{equation}~\label{eq1maxresseptn1}
  \begin{cases}
    \irr_{\max} > \floor*{\frac{2m}{n}} + \ceil*{\frac{2n}{m}} + 2^{\alpha} & \text{if } d_n \leqslant 20, \\
    \irr_{\max} < \ceil*{\frac{2n}{m}} + 2^{\beta} & \text{if } d_n > 3.
  \end{cases}
\end{equation}
\end{proposition}
\begin{proof}
Let $\mathscr{D}=(d_1,d_2,\dots,d_n)$ be a degree sequence with $d_n > d_{n-1} > \dots > d_2 > d_1$ and $\Delta = d_n$. Assume $\irr_{\max} < \Delta(\Delta - 1)^2$ by considering $\Delta(\Delta - 1)^2 < m(m+1)$, and suppose $\beta > \alpha > d_1$. Equation~\eqref{eq2maxresseptn1} asserts that the sum of the greatest integer less than or equal to $2m/n$ and the smallest integer greater than or equal to $2n/m$ is
\begin{equation}~\label{eq2maxresseptn1}
    \floor*{\frac{2m}{n}} + \ceil*{\frac{2n}{m}} = 4.
\end{equation}
Now, consider relation~\eqref{eq1maxresseptn1} in terms of the values of $\alpha$ and $\beta$, taking into account the condition $\beta > \alpha > d_1$. We find that the first two elements of $\mathscr{D}$ satisfy $d_1 + d_2 =$ [value missing], and from~\eqref{eq2maxresseptn1}, for the last two elements of $\mathscr{D}$, the term $\alpha d_1 + \beta d_2$ satisfies 
\begin{equation}~\label{eq3maxresseptn1}
  \alpha d_1 + \beta d_2 >  \floor*{\frac{2n\Delta}{m}} + \ceil*{\frac{2\Delta m}{n}}.
\end{equation}
Hence, from~\eqref{eq2maxresseptn1}, we obtain rough bounds for the values of $2m/n$ and $2n/m$. Therefore, from~\eqref{eq3maxresseptn1} we deduce that for the term $\alpha d_1 + \beta d_2 + \alpha d_3 + \beta d_4 + \dots + \alpha d_{n-1} + \beta d_n$, where
\[
\alpha d_1 < \beta d_2 < \alpha d_3 < \dots < \alpha d_{n-1} + \beta d_n,
\]
and according to $2^\alpha m(m+1) \leqslant \alpha^4 \irr_{\min}(T)$, it follows that
\[
2m(m+1) \leqslant \alpha^2 \irr_{\max}(T).
\]
Then, the upper bound of the maximum value of the Albertson index is given by
\begin{equation}~\label{eq4maxresseptn1}
 \irr_{\max} < \alpha \sum_{i=1}^{n} d_{2i-1} + \beta \sum_{i=1}^{n} d_{2i} + (\Delta - 3)^2.
\end{equation}
Hence, from~\eqref{eq4maxresseptn1} for $d_1 \geqslant 10$ and $d_n \gg d_1$ we notice that $2^{\alpha} > \irr_{\max}$. Thus, the relation~\eqref{eq1maxresseptn1} holds for $\alpha$, and according to the term $2^{\beta} > 2^{\alpha}$ satisfying $2^{\beta} > \irr_{\max}$, the relation~\eqref{eq1maxresseptn1} also holds for $\beta$ when $d_1 > 3$. As desired.
\end{proof}

For each degree sequence $\irr_{\max}$ and $\irr_{\min}$ values are given and are relatively close, indicating that for these graphs. Through Table~\ref{tab01maximumandminimum}, we noticed for a given the degree sequence, the maximum and minimum irredundance numbers indicate the variety and bounds of the irredundant vertex sets in terms of size.In fact, both $\alpha$ and $\beta$ represent theoretical bounds or empirical parameters to predict or estimate irredundance based on the degrees.

\begin{table}[H]
    \centering
\begin{tabular}{|l|c|c|c|c||l|c|c|c|c|}
 \hline
Degree	   & $\irr_{\max}$	 &  $\irr_{\min}$ &	$\alpha$	& $\beta$ & Degree	   & $\irr_{\max}$	 &  $\irr_{\min}$ &	$\alpha$	& $\beta$ \\ \hline
$(4,5,6,7,11)$     &    268  &	 274 & 	4.5 & 	9 & $(6,6,7,10,12)$    &	394	 &   392 &	6 &	11 \\ \hline
$(7,8,9,13,14)$    &	598	 &  592	 &	7.5 &	13.5 & $(9,10,11,15,18)$  &	898  &	896	 &	9.5	& 16.5 \\ \hline
$(10,12,13,17,22)$ &	1242 &	1244 &	11&	19.5 & $(12,14,15,19,26)$ &	1666 &	1672 &	13 & 22.5 \\ \hline
$(14,16,17,21,30)$ &	2154 &	2164 &	15	& 25.5 & $(16,18,19,23,34)$ &	2706 &	2720 &	17	& 28.5 \\ \hline
$(18,20,21,25,38)$ &	3322 &	3340 &	19	& 31.5 & $(20,22,23,27,42)$ &	4002 &	4024 &	21 &	34.5 \\ \hline
\end{tabular}
\caption{The maximum and minimum value of Albertson index with $\mathscr{D}=(d_1,d_2,d_3,d_4,d_5)$.}
    \label{tab01maximumandminimum}
\end{table}
For a degree sequence $\mathscr{D}=(d_1,d_2,\dots,d_n)$, let $\lambda$ be the average of elements $\mathscr{D}$. Then, we noticed that among~\eqref{resultvasn1} the relationship between the number of vertices and $\lambda$ had provided as
\begin{equation}~\label{resultvasn1}
n=\sum_{i=1}^{n}(2d_i-\lambda).
\end{equation}
Proposition~\ref{maxresseptn2} establishes that for any tree $T$ of order $n$, 
the difference between the minimum and maximum values of the Albertson index is bounded above by a constant $\lambda$.
Furthermore, Lemma~\ref{lem1maxvalue} refines this by showing that for a tree with degree sequence 
$\mathscr{D}=(d_1,d_2,\dots,d_n)$, the maximum Albertson index satisfies the inequality~\eqref{eq1lem1maxvalue}, thus providing an explicit upper bound in terms of the maximum degree and parameters related to the structure of $T$. 
\begin{proposition}~\label{maxresseptn2}
Let $T$ be a tree of order $n$, $\irr_{\min}$ be the minimum value ($\irr_{\max}$ be the maximum value, respectively) of the Albertson index. Then, the difference between $\irr_{\min}$ and  $\irr_{\max}$ satisfy $D(\irr_{\min},\irr_{\max})\leqslant \lambda$.
\end{proposition}

\begin{lemma}~\label{lem1maxvalue}
Let $T$ be a tree of order $n$ given in~\eqref{resultvasn1} with degree sequence $\mathscr{D} = (d_1, d_2, \dots, d_n)$. Then, the maximum value of the Albertson index satisfies
\begin{equation}~\label{eq1lem1maxvalue}
\frac{\sqrt{m(n-3)^2} \irr_{\max}}{(\Delta - 1)^4} < \lambda.
\end{equation}
\end{lemma}

\begin{proof}
Assume $\mathscr{D} = (d_1, d_2, \dots, d_n)$ is a degree sequence with $d_n \geqslant d_{n-1} \geqslant \dots \geqslant d_2 \geqslant d_1$, and the number of vertices is given in~\eqref{resultvasn1}. We consider three terms: $\sqrt{m(n-3)^2}$, $\irr_{\max}$, and $(\Delta - 1)^4$, in relation to the value of $\lambda$. We analyze their relationships to prove the validity of relation~\eqref{eq1lem1maxvalue} connecting these terms.

Since clearly $\irr_{\max} > \lambda$, for the maximum value of the Albertson index, the relation
\begin{equation}~\label{eq2lem1maxvalue}
\sqrt{m(n-3)^\lambda} > \lambda (\Delta - 1)^4
\end{equation}
holds, where $\lambda > 2$.

Because $\lambda > 2$ satisfies $\lambda n^2 < n^{(\lambda + 1)/2}$, and considering the value of $\lambda$ yields $2^{\lambda/2} > 4 (n - 1)$, it follows that $2^\lambda > \Delta$. Thus, we observe that
\[
\sqrt{\lambda(n-3)} \irr_{\max} > \lambda (\Delta - 1)^4,
\]
where $(n-3)^\lambda > m (n-3)^2$.

Then, from~\eqref{eq2lem1maxvalue} we find that
\begin{equation}~\label{eq13em1maxvalue}
0 \leqslant \frac{\sqrt{m(n-3)^2}}{(\Delta - 1)^4} < \frac{\lambda}{4}.
\end{equation}

Inequality~\eqref{eq13em1maxvalue} holds if $m \geqslant 0$, $\Delta \neq 1$, and $\lambda > 0$. Hence, the upper bound is
\[
m \cdot (n-3)^2 < \frac{\lambda^2}{16} \cdot (\Delta - 1)^8.
\]
Therefore, from~\eqref{eq13em1maxvalue} the quantity associated with the maximum values of the Albertson index satisfies $\irr_{\max} > \lambda / 4$, and thus equation~\eqref{eq1lem1maxvalue} is satisfied.
\end{proof}

Through Figure~\ref{fig01practical}, we present a tree representing a practical application of the Albertson index (by considering Figure~\ref{fig01startreecond01}) by providing the maximum value of the Albertson index.
\begin{figure}[H]
    \centering
\begin{tikzpicture}[scale=.8]
\node[circle, fill=black, draw, inner sep=0pt, minimum size=2pt] (A) at (0,0) {};
\node[circle,fill=black, draw, inner sep=0pt, minimum size=2pt] (B) at (-4,-2) {};
\node[circle,fill=black, draw, inner sep=0pt, minimum size=2pt] (C) at (0,-2) {};
\node[font=\footnotesize] at (5.7,-4.2) {$\iddots$};
\node[circle, fill=black, draw, inner sep=0pt, minimum size=2pt] (D) at (4,-2) {};
\node[font=\footnotesize] at (4,-4.3) {$\vdots$};
 
\node[circle, fill=black, draw, inner sep=0pt, minimum size=2pt] (E) at (-6,-4) {};
\node[font=\footnotesize] at (2.3,-4.3) {$\vdots$};
\node[circle, fill=black, draw, inner sep=0pt, minimum size=2pt] (F) at (-4,-4) {};
\node[font=\footnotesize] at (2,-4.3) {$\vdots$};
\node[circle, fill=black, draw, inner sep=0pt, minimum size=2pt] (G) at (-2,-4) {};
\node[font=\footnotesize] at (0,-4.3) {$\vdots$};
 
\node[circle, fill=black, draw, inner sep=0pt, minimum size=2pt] (H) at (-1.6,-4) {};
\node[font=\footnotesize] at (-1.6,-4.3) {$\vdots$};
\node[circle, fill=black, draw, inner sep=0pt, minimum size=2pt] (I) at (0,-4) {};
\node[font=\footnotesize] at (-5.7,-4.2) {$\ddots$};
\node[circle,  fill=black, draw, inner sep=0pt, minimum size=2pt] (J) at (2,-4) {};
\node[font=\footnotesize] at (-2,-4.3) {$\vdots$};
 
\node[circle, fill=black, draw, inner sep=0pt, minimum size=2pt] (K) at (2.3,-4) {};
\node[circle, fill=black, draw, inner sep=0pt, minimum size=2pt] (L) at (4,-4) {};
\node[circle, fill=black, draw, inner sep=0pt, minimum size=2pt] (M) at (6,-4) {};
\node[font=\footnotesize] at (-3.7,-4.2) {$\ddots$};
 
\node[circle, fill=black, draw, inner sep=0pt, minimum size=2pt] (A1) at (-7,-6) {};
\node[circle,  fill=black, draw, inner sep=0pt, minimum size=2pt] (B1) at (-5,-6) {};
\node[circle, fill=black, draw, inner sep=0pt, minimum size=2pt] (T1) at (7,-6) {};
\node[font=\footnotesize] at (-5.7,-4.2) {$\ddots$};
 
\node[circle, fill=black,  draw, inner sep=0pt, minimum size=2pt] (A1_1) at (-8,-8) {};
\node[circle, fill=black, draw, inner sep=0pt, minimum size=2pt] (A1_2) at (-7.5,-8) {};
\node[circle, fill=black, draw, inner sep=0pt, minimum size=2pt] (A1_3) at (-7,-8) {};
\node[circle, fill=black, draw, inner sep=0pt, minimum size=2pt] (A1_4) at (-6.3,-8) {};

\node[circle, fill=black, draw, inner sep=0pt, minimum size=2pt] (B1_1) at (-6,-8) {};
\node[circle, fill=black, draw, inner sep=0pt, minimum size=2pt] (B1_2) at (-5.5,-8) {};
\node[circle, fill=black, draw, inner sep=0pt, minimum size=2pt] (B1_3) at (-4.5,-8) {};
\node[circle, fill=black, draw, inner sep=0pt, minimum size=2pt] (B1_4) at (-4,-8) {};

\node[circle, fill=black, draw, inner sep=0pt, minimum size=2pt] (T1_1) at (6,-8) {};
\node[circle, fill=black, draw, inner sep=0pt, minimum size=2pt] (T1_2) at (6.5,-8) {};
\node[circle, fill=black, draw, inner sep=0pt, minimum size=2pt] (T1_3) at (7.5,-8) {};
\node[circle, fill=black, draw, inner sep=0pt, minimum size=2pt] (T1_4) at (8,-8) {};

\node[font=\footnotesize] at (-6,-8.3) {$\vdots$};
\node[font=\footnotesize] at (-6.3,-8.3) {$\vdots$};
\node[font=\footnotesize] at (-7,-8.3) {$\vdots$};
\node[font=\footnotesize] at (-7.3,-8.3) {$\vdots$};
\node[font=\footnotesize] at (-7.7,-8.3) {$\vdots$};
\node[font=\footnotesize] at (-8,-8.3) {$\vdots$};
\node[font=\footnotesize] at (-5.5,-8.3) {$\vdots$};
\node[font=\footnotesize] at (-4.5,-8.3) {$\vdots$};
\node[font=\footnotesize] at (-4,-8.3) {$\vdots$};

\node[font=\footnotesize] at (6,-8.3) {$\vdots$};
\node[font=\footnotesize] at (6.5,-8.3) {$\vdots$};
\node[font=\footnotesize] at (7.5,-8.3) {$\vdots$};
\node[font=\footnotesize] at (8,-8.3) {$\vdots$};
 
\draw (A) -- (B);
\draw (A) -- (C);
\draw (A) -- (D);

\draw (B) -- (E);
\draw (B) -- (F);
\draw (B) -- (G);

\draw (C) -- (H);
\draw (C) -- (I);
\draw (C) -- (J);

\draw (D) -- (K);
\draw (D) -- (L);
\draw (D) -- (M);

\draw (E) -- (A1);
\draw (F) -- (B1);
\draw (M) -- (T1);

\draw (A1) -- (A1_1);
\draw (A1) -- (A1_2);
\draw (A1) -- (A1_3);
\draw (A1) -- (A1_4);

\draw (B1) -- (B1_1);
\draw (B1) -- (B1_2);
\draw (B1) -- (B1_3);
\draw (B1) -- (B1_4);

\draw (T1) -- (T1_1);
\draw (T1) -- (T1_2);
\draw (T1) -- (T1_3);
\draw (T1) -- (T1_4);

\node[font=\footnotesize] at (0,-8) {$\dots$};
\end{tikzpicture}
    \caption{A practical application of the Albertson index.}
    \label{fig01practical}
\end{figure}
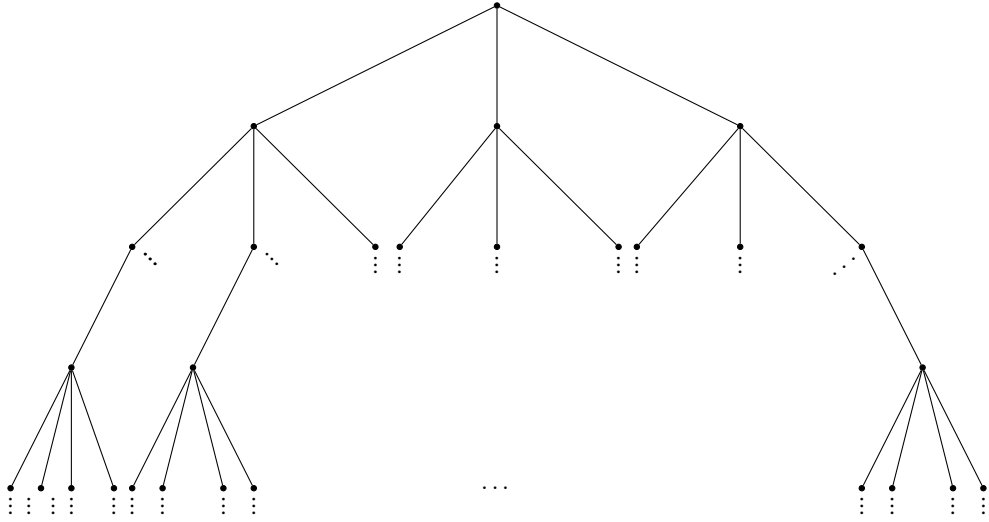

In fact, through Lemma~\ref{lem2irrvalue} establishes a lower bound on the Albertson index of a tree based on its degree sequence.

\begin{lemma}~\label{lem2irrvalue}
Let $T$ be a tree of order $n$ with degree sequence $\mathscr{D}=(d_1,d_2,\dots,d_n)$ where $d_n\geqslant d_{n-1}\geqslant \dots \geqslant d_2\geqslant d_1$. Then, the Albertson index satisfy
\begin{equation}~\label{eq1lem2irrvalue}
\frac{2n\sqrt{2n}}{4}<\frac{(n^2-2m)\Delta\sqrt{2mn}}{2m\lambda\sqrt{2m}}\leqslant \irr(T).
\end{equation}
\end{lemma}
\begin{proof}
Assume a degree sequence $\mathscr{D}=(d_1,d_2,\dots,d_n)$ where $d_n\geqslant d_{n-1}\geqslant \dots \geqslant d_2\geqslant d_1$. Clearly, $\Delta\leqslant \lambda \varphi$ where we consider $\varphi \approx 1.66,$ notice that $d_n-d_{n-1}>d_2$, $d_{n-1}-d_{n-2}>d_2$,$\dots,$ which indicate sufficient irregularity in the degree sequence. Thus, by considering the constant term is $d_n^2+d_1^2+d_n-d_1-2n+2$, which is used to form a lower bound on the Albertson index. Then, the Albertson index satisfy
\begin{equation}~\label{eq2lem2irrvalue}
 2n\sqrt{2n}<\lambda \irr(T).
\end{equation}
This relation leads us to, through the ratio achieved for the arithmetic mean of the elements of the given degree sequence $\lambda$ and the maximum degree $\Delta$ as $n^2\leqslant\frac{ 3\Delta}{\lambda}$. It follows that an upper bound on $\irr(T)$ satisfy 
\begin{equation}~\label{eq3lem2irrvalue}
 \irr(T) \leqslant \frac{2\lambda mn^2-4\Delta n-n^3 }{4m\sqrt{2m}}.
\end{equation}
Therefore, by combining the inequalities \eqref{eq2lem2irrvalue} and \eqref{eq3lem2irrvalue}, we deduce 
$\sqrt{2\lambda mn^2}>\irr(T)$. Then, 
\[
4\Delta (d_n-d_{n-1})<(n^2-2m)\Delta,
\]
where is satisfy  the relationship $(n^2-2m)\Delta\sqrt{2mn} > \irr(T)+4\Delta (d_n-d_{n-1})$. Thus,
\begin{equation}~\label{eq4lem2irrvalue}
(n^2-2m)\Delta\sqrt{2mn} > \irr(T)+4\Delta (d_{n-k}-d_k),
\end{equation}
where $\Delta <k<n$.  Then, we noticed that 
\begin{equation}~\label{eq5lem2irrvalue}
\frac{(n^2-2m)\Delta\sqrt{2mn}}{2m\lambda\sqrt{2m}} \leqslant \irr(T).
\end{equation}
Through this relation, the upper bound has been proven. From \eqref{eq5lem2irrvalue} establishes a further lower bound for the Albertson index. Hence, for the left-hand side, by testing the ratio of the number of vertices to the number of edges, it satisfies $2\Delta n> m/n$ and $2n\sqrt{2n}>2\Delta$. Then, according to~\eqref{eq4lem2irrvalue} and \eqref{eq5lem2irrvalue} we find that 
\begin{equation}~\label{eq6lem2irrvalue}
2n\sqrt{2n}<(n^2-2m)\Delta\sqrt{2mn}.
\end{equation}
Thus, 
\begin{equation}~\label{eq7lem2irrvalue}
n\sqrt{2n} \leqslant 2\lambda mn^2-4\Delta n.
\end{equation}
Finally, Combining inequalities~\eqref{eq2lem2irrvalue}--\eqref{eq7lem2irrvalue} we find that it proves the main claim~\eqref{eq1lem2irrvalue} holds. As desire.
\end{proof}

\section{The Effect of Bounds on Albertson Index}\label{sec5}

Theorem \ref{tythmAlbn1} establishes a lower bound for the Albertson index of a tree $T$ with $n$ vertices and $m$ edges, characterized by its degree sequence $\mathscr{D} = (d_1, d_2, \dots, d_n)$ arranged in non-decreasing order. The average degree $\lambda$ multiplied by $n$, and a term depending on the product of the minimum and maximum degrees $\delta$ and $\Delta$.

\begin{theorem}~\label{tythmAlbn1}
Let $T$ be a tree with $n$ vertices and $m$ edges, having degree sequence $\mathscr{D} = (d_1, d_2, \dots, d_n)$ where $d_n \geqslant d_{n-1} \geqslant \dots \geqslant d_2 \geqslant d_1$, minimum degree $\delta(T)$, and maximum degree $\Delta(T)$. Then, the Albertson index satisfies
\begin{equation}~\label{eq1tythmAlbn1}
\irr(T) \geqslant \sum_{i=2}^{n} d_i^2 - d_1 + 3 m n - 2 \floor*{\frac{3 n^2 - 1}{2}} + \frac{3}{4} \lambda n - 5 \delta \Delta.
\end{equation}
\end{theorem}
\begin{proof}
Assume a tree $T$ with degree sequence $\mathscr{D} = (d_1, d_2, \dots, d_n)$ where $d_n \geqslant d_{n-1} \geqslant \dots \geqslant d_2 \geqslant d_1$, minimum degree $\delta(T)$, and maximum degree $\Delta(T)$. According to Lemma~\ref{le.alb5}, we have 
\begin{equation}~\label{eq2tythmAlbn1}
\irr(T) > \sqrt{2 m n + n \delta (n - 1)}.
\end{equation}
Thus, the term $2 m n < n \delta (n - 1)$ satisfies $2 m n < 2 \Delta (\Delta - 1)^2 < n \delta (n - 1)$ where $\delta > 2$ and $\lambda n > \Delta$. Hence, from~\eqref{eq2tythmAlbn1} for $n \geqslant 4$, indicating that stars with sufficiently many vertices satisfy the condition, possibly due to their high irregularity. Therefore, by considering the inequality $2 \Delta < (n-1)^2$, we have
\begin{equation}~\label{eq3tythmAlbn1}
\floor*{\frac{3 n^2 - 1}{2}} < 2 \Delta (\Delta -1)^2.
\end{equation}
Thus, we observe that $\irr(T) > \frac{3}{4} \lambda n - 5 \delta \Delta$. Then, by considering $3 m n$ and the fact that $m = n - 1$, it follows that $3 m n = 3 n^2 - 3 n$. Then,
\begin{equation}~\label{eq4tythmAlbn1}
\sum_{i=2}^{n} d_i^2 - d_1 + 3 m n - 2 \floor*{\frac{3 n^2 -1}{2}} > \frac{3}{4} \lambda n - 5 \delta \Delta.
\end{equation}
Therefore, through the discussion presented in the relations~\eqref{eq3tythmAlbn1} and \eqref{eq4tythmAlbn1}, it follows from the comparison of the two quantities in~\eqref{eq3tythmAlbn1} that relation~\eqref{eq1tythmAlbn1} holds. Thus, we find that
\begin{equation}~\label{eq5tythmAlbn1}
\irr(T) > \sum_{i=2}^{n} d_i^2 - d_1 + 3 m n - 2 \floor*{\frac{3 n^2 -1}{2}}.
\end{equation}
Finally, by considering $\irr(T) > \frac{3}{4} \lambda n - 5 \delta \Delta$, the inequalities~\eqref{eq2tythmAlbn1}--\eqref{eq5tythmAlbn1} imply~\eqref{eq1tythmAlbn1}.
\end{proof}

According to Theorem~\ref{tythmAlbn1}, its significance stems from regulating the degree distribution and spectral characteristics of trees (especially through~\eqref{eq4tythmAlbn1}), probably to promote significant irregularity (e.g., a high Albertson or Sigma index) or to prioritize trees with elevated maximum degrees, such as star graphs. This concept is pertinent in fields like chemical graph theory, network optimization, and extremal graph theory. Providing more details about $\lambda$ or the specific tree type could enhance the analysis.\par 
To assess the influence of the Albertson index on graphs and trees based on the discussion presented in Table~\ref{tab01maximumandminimum}, we illustrate in Figure~\ref{fig01aconnectedgraph} the Albertson index applied to a connected graph, highlighting the extent of its effect.
\begin{figure}[H]
    \centering
    \includegraphics[width=0.9\linewidth]{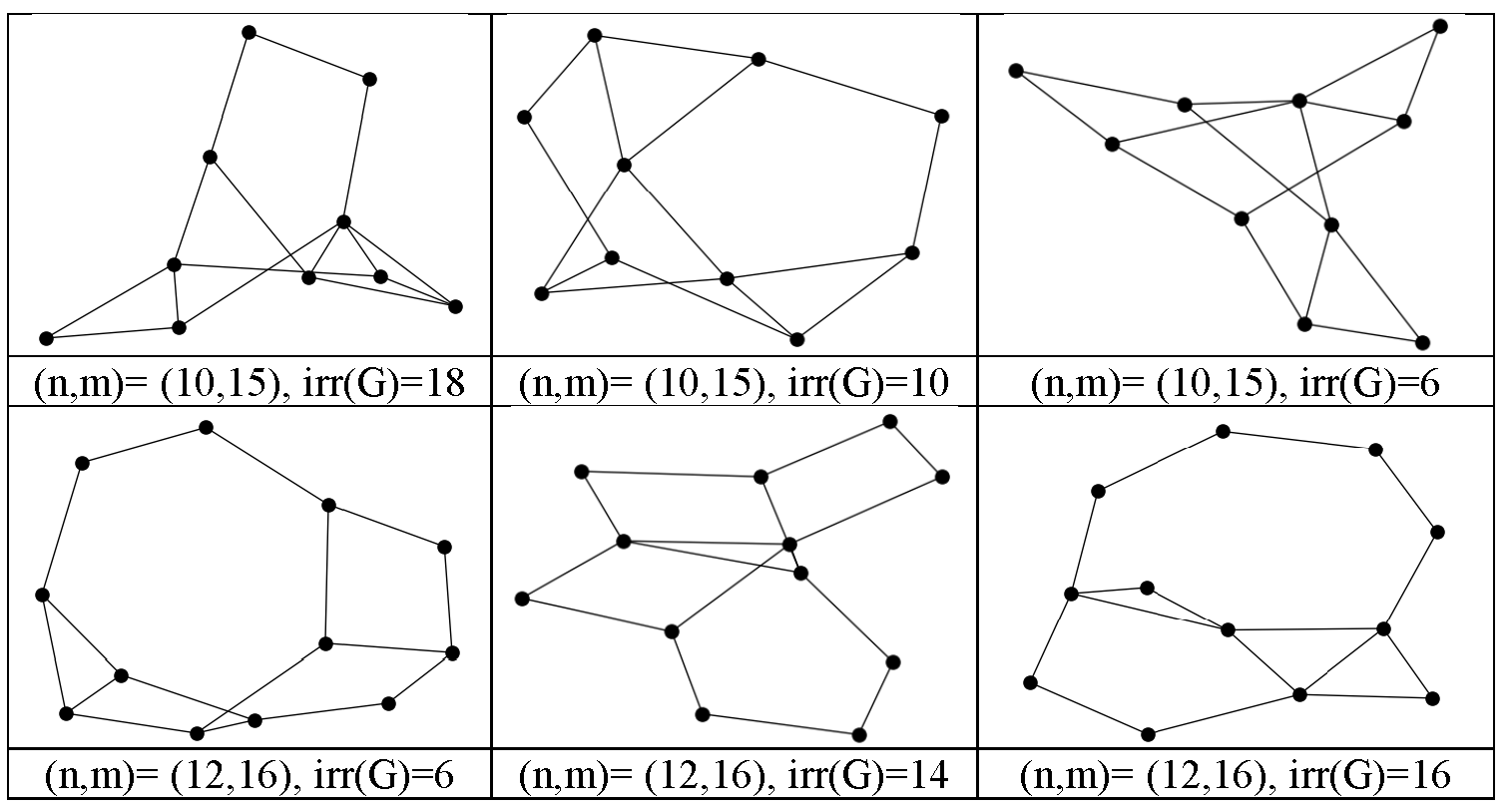}
    \caption{Albertson index among a connected graph.}
    \label{fig01aconnectedgraph}
\end{figure}

Theorem~\ref{tythmAlbn2} establishes a lower bound for the Albertson index of a tree with degree sequence $\mathscr{D} = (d_1, d_2, \dots, d_n)$ arranged in non-increasing order, where $\Delta(T)$ denotes the maximum degree, the Albertson index $\irr(T)$ satisfies the inequality~\eqref{eq1tythmAlbn2}, where $\lambda$ is the average degree of the vertices in $\mathscr{D}$. This result quantifies how the degree distribution and maximum degree constrain the irregularity measured by the Albertson index.

\begin{theorem}~\label{tythmAlbn2}
Let $T$ be a tree with $n$ vertices and $m$ edges, and degree sequence $\mathscr{D}=(d_1,d_2,\dots,d_n)$ where $d_n \geqslant d_{n-1} \geqslant \dots \geqslant d_2 \geqslant d_1$, and let $\Delta(T)$ denote its maximum degree. Then, the Albertson index satisfies
\begin{equation}~\label{eq1tythmAlbn2}
\irr(T) \geqslant n \left(\frac{m(n+1) - \Delta(n-\Delta) + \floor*{\dfrac{2(m-\Delta)^3}{n-2}}}{\lambda^3} \right).
\end{equation}
\end{theorem}
\begin{proof}
Assume the maximum degree $\Delta = \Delta(T)$ satisfies $\Delta > 16$, and the minimum degree $\delta > 4$. Then, for the degree sequence $\mathscr{D}=(d_1,d_2,\dots,d_n)$ where $d_n \geqslant d_{n-1} \geqslant \dots \geqslant d_2 \geqslant d_1$, we have
\begin{equation}~\label{eq2tythmAlbn2}
\floor*{\dfrac{2(m-\Delta)^2}{n-2}} < (n+1)m - \Delta(n-\Delta).
\end{equation}
From~\eqref{eq2tythmAlbn2}, considering the terms $m(n+1) - \Delta(n-\Delta)$ and $\delta \lambda^2$, we observe that $n \beta \irr + (m \Delta) \geqslant m \beta (n+1) - \Delta \beta (n-\Delta)$. By assuming $\Delta(n-\Delta) \geqslant \delta \lambda^2$, it follows that
\begin{equation}~\label{eq3tythmAlbn2}
\floor*{\dfrac{2(m-\Delta)^2}{n-2}} - m(n+1) - \Delta(n-\Delta) \leqslant \delta \lambda^2.
\end{equation}
Hence, considering the values of $\Delta$ and $\delta$ leads to~\eqref{eq2tythmAlbn2} and \eqref{eq3tythmAlbn2}. Thus, equation~\eqref{eq4tythmAlbn2} provides an upper bound on the normalized quantity
\begin{equation}~\label{eq4tythmAlbn2}
\frac{1}{\lambda^2} \left( \floor*{\dfrac{2(m-\Delta)^2}{n-2}} - m(n+1) - \Delta(n-\Delta) \right) \leqslant \frac{2n-2}{2\delta}.
\end{equation}
Moreover, for the term $m(n+1)$ we have $m(n+1) = n^2 - 1$, and for the term $\Delta(n-\Delta)$, the inequality holds when $2\Delta \leqslant n$ and $\lambda < \Delta$. Then, since $2n - 2 > 2\delta$ and $2n > \lambda^2$, we observe that $\Delta(n-\Delta) > \lambda^2$. According to~\eqref{eq4tythmAlbn2}, for $\lambda^3$ it follows that $0 < \frac{n}{\lambda^3} < 1$. By considering $\lambda > \delta$, we obtain
\begin{equation}~\label{eq5tythmAlbn2}
\irr(T) \geqslant \frac{n-1}{2\delta} \left( m(n+1) - \Delta(n-\Delta) + \floor*{\dfrac{2(m-\Delta)^3}{n-2}} \right).
\end{equation}
Therefore, from~\eqref{eq2tythmAlbn2}--\eqref{eq5tythmAlbn2} we conclude that~\eqref{eq1tythmAlbn2} holds.
\end{proof}
Theorem~\ref{tythmAlbn2} establishes a fundamental lower bound on the Albertson irregularity index $\irr(T)$ of a tree $T$ by consider the average degree $\lambda$ of its degree sequence.  The theorem is significant for extremal graph theory as it facilitates comparisons among trees and supports bounding the Albertson index when exact calculation is difficult. Moreover, it has practical implications in chemical graph theory and network analysis, where the Albertson index correlates with molecular stability and network heterogeneity. This bound enriches the theoretical understanding and broadens applications of degree-based topological indices in graph theory. 
Through Theorem~\ref{tythmAlbn2}, we present in Corollary~\ref{corollarynumbern1} the lower bound of the Albertson index.
\begin{corollary}~\label{corollarynumbern1}
Let $T$ be a tree with $n$ vertices and $m$ edges, according to Theorem~\ref{tythmAlbn2}. Then, 
\begin{equation}~\label{eq1corollarynumbern1}
\irr(T)\geqslant \frac{4m^2}{n}+\floor*{\frac{2n}{3}}+\ceil*{\frac{2n+1}{3}}.
\end{equation}
\end{corollary}

Theorem~\ref{tythmAlbn3} consider $\lambda_{\max}$ and $\lambda_{\min}$ denote the maximum and minimum values of $\lambda$, where $\lambda$ is the average degree in the degree sequence $\mathscr{D}=(d_1,d_2,\dots,d_n)$ with $d_n \geq d_{n-1} \geq \cdots \geq d_1$.

\begin{theorem}~\label{tythmAlbn3}
Let $T$ be a tree with $n$ vertices and $m$ edges, and let $\lambda_{\max}$, $\lambda_{\min}$ be the maximum and minimum values of $\lambda$. Then,
\begin{equation}~\label{eq1tythmAlbn3}
\frac{4m^2 - 2m((\Delta + \lambda_{\min})(n - 1))}{3\Delta + n} \leq \irr(T) \leq \frac{m\alpha^3 + \lambda_{\max}(n - 1)}{n + 4 \lambda \delta^2}.
\end{equation}
\end{theorem}
\begin{proof}
Assume $\lambda$ is the average of the degrees in the degree sequence $\mathscr{D} = (d_1, d_2, \dots, d_n)$ where $d_n \geqslant d_{n-1} \geqslant \dots \geqslant d_2 \geqslant d_1$. Then for $\lambda_{\max}$ and $\lambda_{\min}$ as the maximum and minimum values of $\lambda$, we have $3\delta < \lambda_{\min} \leqslant \lambda_{\max} < \Delta$.

According to Theorem~\ref{tythmAlbn2}, we find that $\irr(T) < m(n+1) - \Delta(n - \Delta)$ and clearly $\irr(T) < 4m^2$. Then,
\begin{equation}~\label{eq2tythmAlbn3}
4m^2 \irr(T) > 2m((\Delta + \lambda_{\min})(n - 1)),
\end{equation}
where $2m((\Delta + \lambda_{\min})(n - 1)) > 4m^2$. Thus, from~\eqref{eq2tythmAlbn3} we observe that the left-hand side of the inequality~\eqref{eq1tythmAlbn3} holds, considering that the value $3\Delta + n$ satisfies Albertson index $\irr(T) > 3\Delta + n$.

To prove the right-hand side of the inequality, consider that $\alpha$ is the average of both $d_1$ and $d_2$. We find that this value increases as $m$ increases. Hence,
\begin{equation}~\label{eq3tythmAlbn3}
(\Delta - 1)^3 \irr(T) > m\alpha^3 + \lambda_{\max} (n - 1).
\end{equation}
From~\eqref{eq3tythmAlbn3}, when $\delta \geqslant 2$ we note that $\irr(T) < n + \lambda \delta^2$ and $\irr(T) < m\alpha^3 + \lambda_{\max} (n - 1)$. Then, assuming $\lambda_{\max} \approx \lambda$ satisfies
\begin{equation}~\label{eq4tythmAlbn3}
\irr(T) \leq \frac{m\alpha^3 + \lambda_{\max}(n - 1)}{n + 4 \lambda \delta^2}.
\end{equation}
Thus, from~\eqref{eq4tythmAlbn3} we observe that the right-hand side of the inequality~\eqref{eq1tythmAlbn3} holds.
\end{proof}

Actually, Theorem~\ref{tythmAlbn3} establishes explicit lower and upper bounds for the irregularity $\irr(T)$ of a tree in terms of its number of vertices $n$, edges $m$, maximum degree $\Delta$, it offers a deeper understanding of how a tree's irregularity relates simultaneously to degree sequence. This insight facilitates the classification of trees based on spectral and combinatorial features.

\begin{theorem}~\label{tythmAlbn4}
Let $T$ be a tree with $n$ vertices and $m$ edges. Then, the upper bound for the Albertson index satisfies
\begin{equation}~\label{eq1tythmAlbn4}
\irr(T) \leqslant \alpha \floor*{\frac{2n}{3}} + \beta \ceil*{\frac{2n+1}{3}}.
\end{equation}
\end{theorem}
\begin{proof}
According to Corollary~\ref{corollarynumbern1}, the lower bound of the Albertson index is given in~\eqref{eq1corollarynumbern1}. Thus, we note that
\begin{equation}~\label{eq2tythmAlbn4}
\sqrt{\Delta(\Delta-1)^2} \leqslant \floor*{\frac{2n}{3}} + \ceil*{\frac{2n+1}{3}} < \irr(T).
\end{equation}
Then, from~\eqref{eq2tythmAlbn4} we have
\begin{equation}~\label{eq3tythmAlbn4}
\ceil*{\frac{2n+1}{3}} = \begin{cases}
m + 1, & \text{if } n < 1, \\
m + 2, & \text{if } n \geqslant 1.
\end{cases}
\end{equation}
Therefore, if we want to treat it as a function, it can be generated by
\[
\floor*{\frac{2n}{3}} + \ceil*{\frac{2n+1}{3}} = f(x) = \frac{x^2(x+1) + x^2 + 1}{(x-1)^2 (x^2 + x + 1)}.
\]
It holds for the sum in~\eqref{eq2tythmAlbn4} as
\begin{equation}~\label{eq4tythmAlbn4}
\floor*{\frac{2n}{3}} + \ceil*{\frac{2n+1}{3}} = m + \ceil*{\frac{2n+1}{3}} = 
\begin{cases}
2m + 1, & n < 1, \\
2m + 2, & n \geqslant 1.
\end{cases}
\end{equation}
Thus, from~\eqref{eq3tythmAlbn4} and \eqref{eq4tythmAlbn4} we find that according to~\eqref{eq2tythmAlbn4}, it holds that $\sqrt{\Delta(\Delta - 1)^2} \leqslant n + 1 < \irr(T).$ Then,
\begin{equation}~\label{eq5tythmAlbn4}
\irr(T) \geqslant \alpha \floor*{\frac{2n}{3}} + \sqrt{\Delta(\Delta - 1)^2}, \quad \irr(T) \geqslant \beta \ceil*{\frac{2n+1}{3}} + \sqrt{\Delta(\Delta - 1)^2}.
\end{equation}
Finally, from~\eqref{eq5tythmAlbn4} we conclude that~\eqref{eq2tythmAlbn4} holds.
\end{proof}

\begin{corollary}~\label{corollarynumbern2}
Let $T$ be a tree with $n$ vertices and $m$ edges, according to Theorem~\ref{tythmAlbn4} for $n\in \mathbb{R}$ and $\frac{2n}{3}\notin \mathbb{Z}$. Then, 
\begin{equation}~\label{eq1corollarynumbern2}
\floor*{\frac{2n}{3}} + \ceil*{\frac{2n+1}{3}}=\frac{5}{6}+\frac{2n}{3}+\floor*{\frac{2n}{3}}+\frac{1}{\pi} \sum_{k=1}\left(\frac{1}{k}\sin{\left(\frac{2\pi k+4\pi nk}{3} \right)}  \right).
\end{equation}
\end{corollary}

\section{Conclusion}\label{sec6}
Through this paper, we provide significant new bounds and structural insights regarding the Albertson index for trees based on their degree sequences. We have derived explicit inequalities that tightly constrain both the minimum and maximum values of the Albertson index, connecting them to important graph parameters such as the maximum degree, number of edges, and degree sums and powers.

In particular, Propositions~\ref{resseptn1}–\ref{resseptn3} give lower bounds that relate the minimal Albertson index to degree cubics and quadratics, illustrating the influence of degree irregularity distribution in trees. Propositions~\ref{maxresseptn1} and \ref{maxresseptn2}, together with Lemmas~\ref{lem1maxvalue} and \ref{lem2irrvalue}, further characterize upper bounds and the gap between minimal and maximal index values, highlighting how degree extrema and combinatorial parameters limit the Albertson index variability. The main Theorems~\ref{tythmAlbn1}–\ref{tythmAlbn4} consolidate these findings to provide comprehensive lower and upper bounds for the Albertson index in terms of degree sequences, maximum and minimum degrees, and auxiliary parameters $\lambda, \alpha, \beta$. These bounds integrate various combinatorial components such as floor and ceiling functions, reinforcing the discrete nature of the problem. Corollaries~\ref{corollarynumbern1} and \ref{corollarynumbern2} offer refined numeric and analytic expressions that complement the theoretical results and enable practical computation of Albertson index estimates.

This work advances the understanding of the Albertson index as a metric of graph irregularity by furnishing sharp, degree-sequence-dependent bounds for trees. These results are expected to facilitate further research on extremal problems in topological indices and to support applications where degree irregularity must be quantitatively controlled or estimated. Future extensions may consider more general graph classes or explore algorithmic approaches for exact index computation based on these theoretical foundations.

\section*{Declarations}
\begin{itemize}
	\item Funding: Not Funding.
	\item Conflict of interest/Competing interests: The author declare that there are no conflicts of interest or competing interests related to this study.
	\item Ethics approval and consent to participate: The author contributed equally to this work.
	\item Data availability statement: All data is included within the manuscript.
\end{itemize}


\begin{thebibliography}{99}
\bibitem{abdo2019graph} H.~Abdo, D.~Dimitrov, I.~Gutman, \emph{Graph irregularity and its measures}, Applied Mathematics and Computation. 357 (2019), 317–324.

\bibitem{abdo2014total} H.~Abdo, S.~Brandt, D.~Dimitrov, \emph{The total irregularity of a graph}, Discrete Mathematics \& Theoretical Computer Science. 16(1) (2014), 201–206.

\bibitem{albertson1997irregularity} M.~O. Albertson, \emph{The irregularity of a graph}, Ars Combinatoria. 46 (1997), 219–225.

\bibitem{Bollob1998b} B.~Bollob\'as, \emph{Modern graph theory}, Springer Science \& Business Media. 184 (1998).

\bibitem{broutin2012asymptotics} N.~Broutin, J.-F. Marckert,
 \emph{Asymptotics of trees with a prescribed degree sequence and applications}, Random Structures \& Algorithms. 44(3) (2014), 290–316.
 
\bibitem{Cvetkovi2010Rowlinson} D.~M.~Cvetkovi\'c, P.~Rowlinson, \& S.~Simi\', \emph{An introduction to the theory of graph spectra}, Cambridge: Cambridge university press. 75 (2010). 

\bibitem{Dimitrov2023GaoLin} D.~Dimitrov, W.~Gao, W.~Lin, \& J.~Chen, \emph{Extremal trees with fixed degree sequence for $\sigma$-irregularity}, Discrete Mathematics Letters. 12 (2023), 166-172, doi: 10.47443/dml.2023.106.

\bibitem{dorjsembe2022irregularity} S.~Dorjsembe, L.~Buyantogtokh, K.~C. Das, B.~Horoldagva, \emph{Graphs with maximum irregularity}, Computational and Applied Mathematics. 41(6) (2022), Article 222.

\bibitem{Fiedler1973} M.~Fiedler, \emph{Algebraic connectivity of graphs}, Czechoslovak Mathematical Journal. 23(2) (1973), 298-305.

\bibitem{Ghalavand23Gutman} A.~Ghalavand, I.~Gutman, \& M.~Tavakoli, \emph{Irregularity measure of graphs}, Journal of Mathematics.  1 (2023), 4891183.

\bibitem{Gutman72Moc} I.~Gutman, N.~Trinajstic,  \emph{Graph theory and molecular orbitals III}, Total-electron energy of alternant hydrocarbons, Chem. Phys. Lett., 17 (1972), pp. 535-538.

\bibitem{HamoudwithDuaa} J.~Hamoud, D.~Abdullah, \emph{Topological Indices with Degree Sequence $\mathscr{D}$ of Tree}, \url{https://arxiv.org/abs/2503.12909v2}

\bibitem{molloy1995critical} M.~Molloy, B.~Reed, \emph{A critical point for random graphs with a given degree sequence}, Random Structures \& Algorithms.
6(2–3) (1995), 161–180.

\bibitem{Mandal2022Prvanovic}
Y.~C.~Mandal,  M.~Prvanovic, \emph{Inverse problem for Albertson irregularity index}, Journal of Algebraic Engineering Mathematics. 12(3) (2022), 1-10.

\bibitem{Munir2023Wus} M.~M.~Munir, U.~T.~Wusqa, \emph{Albertson (Alb) spectral radii and Albertson (Alb) energies of graph operation}, Frontiers in Chemistry. 11 (2023), 1267291.

\bibitem{oboudi2023sombor} M.~R. Oboudi, \emph{On graphs with integer Sombor index}, Journal of Applied Mathematics and Computing. 69 (2023), 193–204.

\bibitem{zhang2022random} P.~Zhang, X.~Wang, \emph{Several topological indices of random caterpillars}, Methodology and Computing in Applied Probability. 24 (2022), 1773–1789.

\bibitem{zhang2012number} X.-M. Zhang, X.-D. Zhang, D. Gray, H. Wang, \emph{The number of subtrees of trees with given degree sequence}, Journal of Graph Theory. 73(3) (2013), 280–295.

\bibitem{YangDeng2023} J.~Yang and H.~Deng,
\emph{The Extremal Sigma Index of Connected Graphs with Given Number of Pendant Vertices}, Research Square, 2023.

\end{thebibliography}
\end{document}